\documentclass[aps,prd,onecolumn]{revtex4-2}
\AtBeginDocument{\fontsize{12.4pt}{14.0pt}\selectfont}

\usepackage[T1]{fontenc}
\usepackage[utf8]{inputenc}
\usepackage{libertinus}
\usepackage{microtype}
\usepackage{amsmath,amssymb,amsthm,mathtools}

\usepackage[
  colorlinks=true,
  linkcolor=blue,
  citecolor=blue,
  urlcolor=blue
]{hyperref}

\usepackage{bookmark}

\usepackage{enumitem}

\usepackage{tikz}
\usetikzlibrary{calc}

\makeatletter
\renewcommand{\section}{%
  \@startsection{section}{1}{\z@}%
    {5pt}
    {5pt}
    {\normalfont\bfseries\centering}%
}
\makeatother

\setlist[itemize]{leftmargin=2em}
\setlist[enumerate]{leftmargin=2em}

\newtheorem{theorem}{Theorem}[section]
\newtheorem{proposition}[theorem]{Proposition}
\newtheorem{lemma}[theorem]{Lemma}

\theoremstyle{definition}
\newtheorem{definition}[theorem]{Definition}

\newcommand{\R}{\mathbb{R}}

\newcommand{\restr}{\upharpoonright}

\usepackage{bm}        
\usepackage{braket}    

\usepackage[
  top=1.65cm,
  bottom=2.30cm,
  left=2.10cm,
  right=2.10cm,
  footskip=1.0cm
]{geometry}

\usepackage{graphicx}

\newcommand{\triLeqTriH}{1.80}
\newcommand{\triLeqTriV}{1.20}
\newcommand{\triGeqTriH}{1.80}
\newcommand{\triGeqTriV}{1.20}

\newcommand{\triLeqDashScale}{1.40}
\newcommand{\triGeqDashScale}{1.30}

\newcommand{\triLeqDashRaise}{-0.32ex}
\newcommand{\triGeqDashRaise}{-1.11ex}

\newcommand{\triLeqKern}{-1.13em}
\newcommand{\triGeqKern}{-1.05em}

\newcommand{\triLeqDashAngle}{-20}
\newcommand{\triGeqDashAngle}{20}

\makeatletter
\newcommand{\trileqslant}{\mathrel{\mathpalette\trileqslant@{}}}
\newcommand{\trileqslant@}[2]{%
  \scalebox{\triLeqTriH}[\triLeqTriV]{$#1\triangleleft$}%
  \kern\triLeqKern%
  \raisebox{\triLeqDashRaise}{%
    \scalebox{\triLeqDashScale}{\rotatebox{\triLeqDashAngle}{$#1-$}}%
  }%
}

\newcommand{\trill}{\mathrel{\mathpalette\trill@{}}}
\newcommand{\trill@}[2]{%
  \scalebox{\triLeqTriH}[\triLeqTriV]{$#1\triangleleft$}%
}

\newcommand{\trigeqslant}{\mathrel{\mathpalette\trigeqslant@{}}}
\newcommand{\trigeqslant@}[2]{%
  \scalebox{\triGeqTriH}[\triGeqTriV]{$#1\triangleright$}%
  \kern\triGeqKern%
  \raisebox{\triGeqDashRaise}{%
    \scalebox{\triGeqDashScale}{\rotatebox{\triGeqDashAngle}{$#1-$}}%
  }%
}

\newcommand{\trigg}{\mathrel{\mathpalette\trigg@{}}}
\newcommand{\trigg@}[2]{%
  \scalebox{\triGeqTriH}[\triGeqTriV]{$#1\triangleright$}%
}
\makeatother

\usepackage[dvipsnames,svgnames,x11names]{xcolor}
\usepackage{tikz}

\usepackage[
  font=small,
  labelfont=bf,
  format=plain,
  width=.75\textwidth,
  justification=RaggedRight
]{caption}
\usepackage{pgfplots}
\pgfplotsset{compat=1.18}

\usepackage[alpine,clock,electronic,geometry,misc,weather]{ifsym}

\newcommand{\C}{\mathfrak{C}}

\newcount\mrows \mrows=12
\newcount\ncols \ncols=8
\newcount\lo \lo=0
\newcount\hi \hi=99

\begin{document}

\makeatletter
\setlength{\parskip}{10pt}
\setlength{\parindent}{20pt}

\renewcommand{\section}{%
  \@startsection{section}{1}{\z@}%
    {5pt}
    {5pt}
    {\normalfont\bfseries\centering}%
}

\def\@hangfrom@section#1#2#3{%
  \@hangfrom{#1#2}#3%
}

\makeatother

\title{\Large Does the Hardin--Taylor Predictor Actually Predict the Future?\\
Yes; but Only in Retrospect}
\author{Ulvi Yurtsever}
\email{ulvihy@gmail.com}
\affiliation{MathSense Analytics, 4834 Stern Avenue, Sherman Oaks, CA 91423}


\date{\today}


\begin{abstract}
The Hardin-Taylor theorem, a consequence of the axiom of choice, gives a
strategy to predict future values of an arbitrary function of time $f:\R\to S$
given its past values. The times in $\R$ at which the strategy fails
are shown to form a countable nowhere dense set for any function $f$.
The theorem is often presented as a striking result
about predicting the future. In this note we discuss the operational content of the theorem and
clarify this frequent interpretation.
The Hardin--Taylor prediction strategy is online: at time $t$ it depends only
on the observed past $f\restr(-\infty,t)$.  Its performance guarantee,
however, is offline: the exceptional set consisting of times at which the
predictor fails is defined only after the entire function has been realized.
We spell out the precise definition of the predictor, discuss
a simple two-function example as a sanity check, and emphasize
the retrospective nature of verifying the predictor's success.
We also examine the proper probabilistic setting
in which the success likelihood of the predictor as a forecasting
procedure should be evaluated. Our conclusion is that the theorem is best
understood as a result about the smallness of the set of times at which the predictor
must be course-corrected,
rather than as a prospectively verifiable procedure for forecasting the future.
\end{abstract}

\maketitle
\vspace{-14pt}

\bookmarksetup{startatroot}

\section{Introduction}

The Hardin-Taylor theorem\cite{HardinTaylor2008, HardinTaylor2013}
is one of the more striking consequences of the axiom of choice.  In one common formulation, the theorem asserts that given the complete past before time $t$ of an arbitrary (not necessarily continuous) function of time $f:\R\to S$, there is a strategy which predicts the present value $f(t)$ correctly for all but countably many times $t$.
A seemingly stronger version states that the strategy
can even produce a continuation agreeing with $f$ on a right-neighborhood of $t$,
again except for a small exceptional set of times $t$.

Many expositions, especially those with a popularizing flavor\cite{Hamkins2025},
make it easy for the reader to take away a stronger operational conclusion
than the theorem actually supports.  In ordinary language, prediction suggests
a prospective ability: at a given time and using data up to that time, one makes
a forecast and has some warrant, available then, that the forecast is likely to be right.
The Hardin-Taylor theorem has a different logic.  Its predictor is indeed defined causally,
using only the past, but the predictor's
construction is continually course-corrected as time rolls along, and the assertion
that it succeeds almost everywhere is verified only retrospectively,
after the entire function has been revealed.

%

The purpose of this note is not to criticize existing expositions or
interpretations of the theorem, but rather to clarify three points that are
often insufficiently explained or emphasized:
\vspace{-10pt}
\begin{enumerate}
    \item the online (synchronous) definition of the predictor;
    \item the offline (asynchronous) definition of its exceptional
    (failure) set relative to the completed function \(f\);
    \item the appropriate probabilistic framework for assessing the
    efficacy of the predictor as a forecasting procedure.
\end{enumerate}
Once these points are examined, the theorem becomes less mysterious.
It is fundamentally a theorem about the well-ordered set of times at which
course corrections to the predictor are necessary, and about the surprising
smallness of this set in \(\mathbb R\). Its dramatic force comes from the
axiom of choice, whose application in the theorem does not by itself confer
any prospective forecasting power.

Our note is closely related in subject matter, but distinct in purpose, from
the recent work of Cox and Elpers\cite{CoxElpers}, who study the robustness
of the Hardin--Taylor predictor under changes of the time coordinate. 

\section{Notation and definition of the predictor}

Let $S$ be a multi-element set of possible states.
In the discussion below one may assume $S=\{0,1\}$ or $S=\R$.  Let $S^\R$ be the set of all total
functions from $\R$ to $S$.
We regard $t\in\R$ as time, and each $f \in S^\R$ as a dynamical system.
For $f\in\ S^\R$, define its past history at time $t$ by
\[
    f_{<t}=f\restr(-\infty,t).
\]
Thus the present value $f(t)$ is not included in the history $f_{<t}$.

Assume the axiom of choice and fix a well-order $\; \trileqslant$ on $S^\R$.  For a partial history $h:(-\infty,t)\to S$, define
\[
    G(h)=\text{the \; $\trileqslant$-least }g\in S^\R \; \text{ such that }\; g\restr(-\infty,t)=h.
\]
For a total function $f\in S^\R$, define
\[
    g_t^f \equiv G(f_{<t}).
\]
Thus the total function $g_t^f$ is the $\trileqslant$-smallest dynamical system compatible
with the past history of $f$ at time $t$, in other words the $\trileqslant$-smallest ``prediction" compatible with $f\restr (-\infty , t)$. Call the total function $s \mapsto g_t^f (s)$ the Hardin-Taylor predictor of $f$ at time $t$: it is equal to $f(s)$ for $s<t$ and targeted (but not guaranteed) to equal $f(s)$ for $s \geqslant t$. Note that the predictor $g_t^f (s)$ is in principle continually redefined (via the dependence on $t$ of the function $g_t^f$) as time $t$ rolls along.

The natural failure/success notion for the predictor is the following:

\begin{definition}[Predictor failure/success]
The Hardin-Taylor predictor $g_t^f$
of $f$ at time $t$ succeeds if there is some $\varepsilon>0$ such that
\[
    g_t^f \restr[t,t+\varepsilon)=f\restr[t,t+\varepsilon).
\]
The times $t$ at which the predictor fails constitute the failure set
\begin{align}
    P_f \equiv \bigl\{t\in\R: \forall \varepsilon>0\;
     \exists s \in [t, t+\varepsilon ) \; \text{ with } \; g_t^f (s) \ne f(s)\bigr\}. \nonumber
\end{align}
\end{definition}
\noindent We can now make the notion of course-correction  more precise:
Suppose the predictor is successful at the time $t=t_0$ but then fails at some later time.
Since by the Hardin-Taylor theorem
(Theorem III.2 below) the failure set $P_f$ is well ordered in $\R$,
there is a time of first failure after $t_0$,
say at $t= t_1$. Thus $t_1 \in P_f$, $ t_1 > t_0$,
and the prediction at $t_0$ remains successful until $t_1$:
$g_{t_0}^f (s) = f(s) \; \forall s \in [t_0, t_1)$.
We show in Lemma III.1 below
(monotonicity of $\trileqslant$-least
predictions compatible with $f$) that whenever $s<t$ we have
$g_s^f\trileqslant g_t^f$, and the strict inequality $g_s^f\trill g_t^f$
holds if there is a failure time $u \in [s,t)$.
This means that not only $\forall s \in [t_0, t_1)$ the prediction $g_{t_0}^f$
is successful: $g_{t_0}^f (s) = f(s)$, but the rolling
predictor $g_s^f$ remains anchored at $g_{t_0}^f$ throughout the interval:
$g_s^f = g_{t_0}^f \;\; \forall s \in [t_0, t_1)$. No redefinition of the
predictor is necessary in this interval. By contrast, since at $t_1$ the predictor does fail,
a course correction (redefinition) {\em is} necessary: for $t > t_1$
the rolling predictor $g_t^f$ is no longer equal to the function $g_{t_0}^f$.
The set $P_f$ is the set of necessary course corrections to the Hardin-Taylor predictor.
It marks the times at which the currently selected prediction
must be discarded and the rolling selector must move upward in the well-order $\trileqslant$.

\section{The Hardin--Taylor mechanism}

The basic monotonicity observation is immediate.

\begin{lemma}[Monotonicity of $\trileqslant$-least predictions compatible with $f$]
If $s<t$, then
\[
    g_s^f\trileqslant g_t^f.
\]
Moreover, if $g_s^f$ fails to agree with $f$ somewhere in $(-\infty,t)$, then
\[
    g_s^f\trill g_t^f.
\]
\end{lemma}

\begin{proof}
If $s<t$, every total function compatible with $f_{<t}$ is also compatible with $f_{<s}$.  Thus the set from which $g_t^f$ is chosen is a subset of the set from which $g_s^f$ is chosen.  Since $g_s^f$ is the least member of the larger set, $g_s^f\trileqslant g_t^f$.  If $g_s^f$ disagrees with $f$ before $t$, then $g_s^f$ is not compatible with $f_{<t}$, so it cannot equal $g_t^f$; hence the inequality is strict.
\end{proof}

This lemma is the whole engine of the theorem.  A failed prediction eliminates the current
$\trileqslant$-least predictor and replaces it with a course-corrected new predictor.

\begin{theorem}[Hardin--Taylor]
For every $f\in S^\R$, the failure set $P_f$ is well-ordered by the usual order
$\leqslant$ on $\R$.
In particular, $P_f$ is countable and nowhere dense.
Equivalently, for every $t\notin P_f$, there is $\varepsilon>0$ such that
\[
    g_t^f\restr(-\infty,t+\varepsilon)=f\restr(-\infty,t+\varepsilon).
\]
\end{theorem}

\begin{proof}[Proof]
Suppose to the contrary that $P_f \subset \R$
is not well ordered. Then we can find $t_0>t_1>t_2>\cdots > t_n > t_{n+1} > \cdots$,
an infinite strictly decreasing sequence of failure times.
Fix $n$.  Since $t_{n+1}\in P_f$, there is no right-neighborhood of
$t_{n+1}$ on which $g_{t_{n+1}}^f$ agrees with $f$.
Taking $\varepsilon=t_n-t_{n+1}$, we find some $s<t_n$ such that
\[
    g_{t_{n+1}}^f(s)\ne f(s).
\]
Because $g_{t_{n+1}}^f$ already agrees with $f$ on $(-\infty,t_{n+1})$, the disagreement occurs after $t_{n+1}$ but before $t_n$.  Hence $g_{t_{n+1}}^f$ is incompatible with the later history $f_{<t_n}$.  By the monotonicity lemma,
\[
    g_{t_{n+1}}^f\trill g_{t_n}^f.
\]
Thus $g_{t_0}^f\trigg g_{t_1}^f\trigg g_{t_2}^f\trigg\cdots \trigg g_{t_n}^f \trigg g_{t_{n+1}}^f
\trigg \cdots $ is an infinite
strictly decreasing sequence in the well-order $\trileqslant$, which is impossible.

Every well-ordered subset $A$ of $(\R , \leqslant )$ is countable because
any element \(a \in A\) has a successor \(a^+\) (set $= \infty$ in case $a \in A$ is largest) and can
be assigned a rational number in \((a,a^+)\), putting $A$ in 1:1 correspondence with
a subset of the rationals. It is nowhere dense as well:
if the closure \(\bar{ A}\) contained a nonempty interval,
then recursively choosing points of \(A\) in successively smaller
subintervals would produce a strictly decreasing sequence in \(A\),
contradicting well-ordering.
\end{proof}

\section{A two-function sanity check}

The following example is useful in conceptualizing the operational content of the theorem.

\begin{proposition}
Let $S=\R$ and define two functions $f_1,f_2 \in \R^\R$ by
\[
    f_1 (x)=
    \begin{cases}
    0,&x < 0,\\
    1,&x \geqslant 0,
    \end{cases}
    \qquad
    f_2 (x)=0\quad\text{for all }x\in\R.
\]
Then $f_{1<0}=f_{2<0}$, hence
\[
    G(f_{1<0})=G(f_{2<0}).
\]
Consequently the predictor must fail at $t=0$ for at least one of $f_1$ and $f_2$.
\end{proposition}

\begin{proof}
Because the histories before $0$ are identical, the predictor receives
the same input for $f_1$ and $f_2$ at time $0$.  Therefore it produces
the same total function, say $g$, as prediction.
But $f_1$ and $f_2$ disagree on every interval $[0,\varepsilon)$.
No single function $g$ can agree with both of them on such an interval.
Thus $0\in P_{f_1}\cup P_{f_2}$.
\end{proof}

This example shows why the exceptional set $P_f$
must be attached to the complete function.
The input to the predictor at time $0$ is the same for $f_1$ and $f_2$,
but the success or failure of the prediction depends
on which completion of the common past is realized.
By contrast, for $t>0$, the histories $f_{1<t}$ and $f_{2<t}$ are no longer the same, since they differ on $[0,t)$.  Thus there is no reason to expect $G(f_{1<t})=G(f_{2<t})$
for positive $t$.

\section{Online prediction and offline verification}

The predictor selection rule
\[
    h \longmapsto G(h)
\]
is online: at time $t$ it depends only on the past $h = f_{<t}$.  But the statement that the prediction succeeds except on a small set is not an online statement.  It is a statement about the completed trajectory of $f$, since the exceptional set $P_f$ of failures
is defined by comparison with the entire function.
At a given time $t$ there is generally no criterion, supplied by the theorem,
telling us whether $t\notin P_f$ or $t\in P_f$.
The theorem's conclusion that the set of failures
is small is revealed after the fact.
This is the sense in which the future is predicted only in retrospect.
The predictor does not use hindsight, but its performance guarantee is formulated and verified with hindsight.

\newpage
\section{How well can the Hardin-Taylor predictor predict the future?}

The smallness of the failure set $P_f$ guaranteed by the Hardin-Taylor theorem is surprising.
It is tempting to interpret this astonishing result as a probabilistic statement
about the efficacy of the predictor: failure is a measure-zero event, success
happens with probability 1 (on a set of full measure). The predictor is as perfect as can be.

But this is the wrong probabilistic approach: the conclusion that
a randomly chosen time in $\R$ falls
outside the failure set with probability 1
is irrelevant as a gauge of the predictor's forecasting ability, because
the failure set is defined only after $f$, the target of the forecast, is known.
In the current light this approach has the flavor of a hat trick.


The correct probabilistic approach should
start by recognizing that the relevant sample space
is the space of functions $S^\R$, not the set of times $\R$.
On one hand, we have the theorem's result about
the smallness of the predictor's retrospective failure
set in $\R$; on the other, we have the prospective question about its performance
as a forecasting procedure.
We attempt to delineate the two using the following formulation:

For definiteness we set $S=\R$.
Fix a well-order $\trileqslant$ on $\R^\R$,  a time \(t_{0}\), and an
observed past history
\[
        h=f\upharpoonright(-\infty,t_{0}).
\]
Let \(\C_h\) be the class of all total functions
\(f:\mathbb R\to\mathbb R\) extending this history, and let
\[
        q_h = G(h)
\]
be the Hardin-Taylor predictor selected using $h$ at time $t_0$.
The proper assessment of
the forecasting power of the predictor demands asking the following question:
what is the probability that the predictor $q_h$ is successful at $t_0$
if the true total function $f$ is picked randomly from $\C_h$?
In plain terms: how likely is the predictor to do well in forecasting
the continuation of $h$?

The answer requires a probability law
\(\mu_h\) on the space \(\C_h\) of possible continuations of the
observed past.  The relevant event is not the event that a randomly chosen
time avoids the exceptional set \(P_f\), but rather the event $A_h \subset \C_h$
defined by the predictor's success
\[
        A_h =
        \left\{
        f\in\C_h :
        \exists \varepsilon>0 \ 
        f \upharpoonright [t_{0},t_{0}+\varepsilon)
        =
        q_h \upharpoonright [t_{0},t_{0}+\varepsilon)
        \right\}.
\]
Thus the forecasting question is whether \(\mu_h(A_h)\) is large. Under
any prior for which agreement with a prescribed continuation
on a right-interval has zero probability, one has\cite{ftnote1}
\[
        \mu_h(A_h)=0.
\]
This would be the case for many standard stochastic
models in which genuine future randomness remains after the past has been fixed.  To be clear,
such a probabilistic failure result does not follow from the Hardin-Taylor theorem itself;
the theorem's analysis is not concerned with possible stochastic models on $\C_h$.

\begin{figure}[t]
\centering

\begin{tikzpicture}[
    x=1cm,y=1cm,
    axis/.style={black!70,thin},
    truef/.style={blue!70!black,thick},
    pred/.style={green!55!black,very thick},
    alt/.style={red!70!black,thick},
    guide/.style={black!60,densely dotted,thin},
    every node/.style={font=\small}
]


\newcommand{\hpast}[1]{%
  1.15*sin(55*((#1)-1.2))-0.55%
}

\pgfmathsetmacro{\yzero}{1.15*sin(55*(0-1.2))-0.55}

\newcommand{\qfuture}[1]{%
  \yzero
  +0.25*(#1)
  +0.95*(1-exp(-0.9*(#1)))
  +0.06*sin(55*(#1))%
}


\begin{scope}[yshift=5.3cm]

\node[font=\bfseries,align=center] at (0,2.4)
{(a) Random time in a fixed world:
 Fixed \(f\), random \(t_0\); almost-sure success.};

\draw[axis] (-4.1,0) -- (4.1,0) node[right] {\(t\)};
\draw[axis] (-3.75,-2.1) -- (-3.75,2.15);


\draw[truef]
    plot[domain=-3.7:0,samples=120,smooth]
    (\x,{\hpast{\x}});

\draw[truef]
    plot[domain=0:3.7,samples=120,smooth]
    (\x,{\qfuture{\x}});

\node[blue!70!black] at (-3.05,0.90) {\(f(t)\)};


\def\xa{-2.70}
\pgfmathsetmacro{\ya}{\hpast{\xa}}

\draw[guide] (\xa,0) -- (\xa,\ya);
\fill (\xa,\ya) circle (1.4pt);

\draw[pred]
    plot[domain=-2.70:-2.10,samples=40,smooth]
    (\x,{\hpast{\x}});

\def\xb{-1.45}
\pgfmathsetmacro{\yb}{\hpast{\xb}}

\draw[guide] (\xb,0) -- (\xb,\yb);
\fill (\xb,\yb) circle (1.4pt);

\draw[pred]
    plot[domain=-1.45:-0.85,samples=40,smooth]
    (\x,{\hpast{\x}});

\draw[guide] (0,0) -- (0,\yzero);
\fill (0,\yzero) circle (1.5pt);
\node[above] at (0,0) {\(t_0\)};

\draw[pred]
    plot[domain=0:0.75,samples=50,smooth]
    (\x,{\qfuture{\x}});

\def\xc{1.45}
\pgfmathsetmacro{\yc}{\qfuture{\xc}}

\draw[guide] (\xc,0) -- (\xc,\yc);
\fill (\xc,\yc) circle (1.4pt);

\draw[pred]
    plot[domain=1.45:2.05,samples=40,smooth]
    (\x,{\qfuture{\x}});

\def\xd{2.65}
\pgfmathsetmacro{\yd}{\qfuture{\xd}}

\draw[guide] (\xd,0) -- (\xd,\yd);
\fill (\xd,\yd) circle (1.4pt);

\draw[pred]
    plot[domain=2.65:3.30,samples=40,smooth]
    (\x,{\qfuture{\x}});

\end{scope}


\begin{scope}[yshift=-1.0cm]

\node[font=\bfseries,align=center] at (0,2.4)
{(b) Random world after a fixed time:
 Fixed \(t_0,h\), random \(f\in\mathcal C_h\); almost-sure failure.};

\draw[axis] (-4.1,0) -- (4.1,0) node[right] {\(t\)};
\draw[axis] (-3.75,-2.1) -- (-3.75,2.15);


\draw[truef]
    plot[domain=-3.7:0,samples=120,smooth]
    (\x,{\hpast{\x}});

\node[blue!70!black] at (-3.05,0.90) {\(h(t)\)};

\draw[guide] (0,0) -- (0,\yzero);
\fill (0,\yzero) circle (1.5pt);
\node[above] at (0,0) {\(t_0\)};


\draw[pred]
    plot[domain=0:3.7,samples=120,smooth]
    (\x,{\qfuture{\x}});

\node[green!45!black] at (3.20,0.55) {\(q_h\)};


\draw[alt]
    (0,\yzero)
    .. controls (0.35,\yzero+0.65) and (0.75,\yzero+1.00) ..
       (1.25,\yzero+1.15)
    .. controls (1.85,\yzero+1.35) and (2.60,\yzero+1.05) ..
       (3.70,\yzero+1.45);

\draw[alt]
    (0,\yzero)
    .. controls (0.40,\yzero+0.35) and (0.85,\yzero+0.35) ..
       (1.35,\yzero+0.60)
    .. controls (2.00,\yzero+0.85) and (2.75,\yzero+0.55) ..
       (3.70,\yzero+0.75);

\draw[alt]
    (0,\yzero)
    .. controls (0.35,\yzero-0.10) and (0.90,\yzero-0.40) ..
       (1.40,\yzero-0.30)
    .. controls (2.00,\yzero-0.10) and (2.75,\yzero-0.50) ..
       (3.70,\yzero-0.40);

\draw[alt]
    (0,\yzero)
    .. controls (0.30,\yzero-0.45) and (0.80,\yzero-0.85) ..
       (1.40,\yzero-1.00)
    .. controls (2.10,\yzero-1.20) and (2.85,\yzero-0.90) ..
       (3.70,\yzero-1.10);

\draw[alt]
    (0,\yzero)
    .. controls (0.25,\yzero+0.10) and (0.60,\yzero-0.15) ..
       (1.10,\yzero-0.05)
    .. controls (1.80,\yzero+0.15) and (2.70,\yzero-0.10) ..
       (3.70,\yzero+0.05);

\end{scope}

\end{tikzpicture}

\caption{
Two contrasting probability thought experiments for the Hardin--Taylor predictor.
(a) A completed function \(f\) is held fixed and the prediction time is
sampled. Each sampled time almost surely lies outside the failure set
\(P_f\), so the predictor (the green curves)
agrees with \(f\) on some right-neighborhood of the sampled times.
(b) Instead, one such time \(t_0\) and its observed past \(h\) are held
fixed. The Hardin--Taylor predictor selects \(q_h\), the green curve that agrees with the original $f$,
while the true continuation almost surely disagrees with $q_h$ if
sampled from \(\mathcal C_h\), represented by the set of red curves. The resolution lies
in the observation that while $t_0 \notin P_f$,
for a randomly sampled continuation $g \in \C_h$ it is almost surely the case
that $t_0 \in P_g$.
}
\label{fig:two-probability-experiments}

\end{figure}
Our conclusion is that the theorem's surprising answer to the question of
how large the failure set \(P_f\subset \mathbb R\) can be, a question that lives in
the time domain \(\mathbb R\), is largely orthogonal to the forecasting
efficacy of the predictor, which is instead a question that lives in the space
\(\mathcal C_h\) of possible continuations compatible with the observed
past.  This distinction---the main takeaway of this note---is perhaps most
vividly seen by placing side by side two probabilistic statements that
appear to conflict but are, in fact, perfectly consistent (FIG.\ 1):

First, for any fixed completed function \(f\), the failure set \(P_f \subset \R \) is
countable and hence Lebesgue null; consequently, if the time \(t_0\) is
sampled according to any probability law \(\Pr\) on \(\mathbb R\) that is
absolutely continuous with respect to Lebesgue measure, then
\[
        \Pr(t_0\in P_f)=0,
\]
in other words, the predictor almost surely succeeds at \(t_0\).

Second, for any fixed time \(t_0\) and observed history \(h\) before \(t_0\),
if the prior \(\mu_h\) satisfies the interval-agreement null condition stated
above, then
\[
        \mu_h(A_h)=0,
\]
and consequently a randomly sampled function \(f\in\C_h\) is almost surely
not in \(A_h\); in other words, the predictor almost surely fails at \(t_0\).

Yet there is no contradiction: the two probabilistic statements are defined
on different sample spaces.  The first samples a time from \(\mathbb R\)
after the completed function \(f\) has been fixed; the second samples a
possible continuation from \(\mathcal C_h\) after the time \(t_0\) and
the observed past history \(h\) have been fixed. The resolution is
that while $t_0 \notin P_f$, we have $t_0 \in P_g$ for $\mu_h$-almost every continuation
$g \in \C_h$. Reconciling these two
statements intuitively in one's mind is essential to understanding why the
Hardin--Taylor theorem by itself does not provide a prospective ability to
forecast the future.


\begin{thebibliography}{9}
\bibitem{ftnote1} For example, suppose that for every rational \(r>0\),
$ \mu_h [ \{f\in \C_h: f\upharpoonright[t_0,t_0+r)
= q_h\upharpoonright[t_0,t_0+r)\} ] =0.$
Then \( \mu_h(A_h)=0 \) since \(A_h\) is the countable union of these
events over \(r\in\mathbb Q_{+}\).
\end{thebibliography}
\end{document}